\documentclass[12pt]{amsart}
 \usepackage{amssymb}
 \usepackage{latexsym}
 \usepackage{graphicx}
 \usepackage{multicol}
 \usepackage{mathrsfs}
 \usepackage{pstricks, pst-node}
 \usepackage{young}
 \usepackage[vcentermath,enableskew]{youngtab}
 \usepackage[all]{xy}
    \SelectTips{cm}{10}     
    \everyxy={<2.5em,0em>:} 
 \usepackage{fancyhdr}      
 \newcounter{ctr}
 \theoremstyle{plain}
 \newtheorem{theorem}{Theorem}[section]
 
 \newtheorem*{lemma*}{Lemma}
 \newtheorem{lemma}[theorem]{Lemma}
 \newtheorem{corollary}[theorem]{Corollary}
 \newtheorem{proposition}[theorem]{Proposition}

 \theoremstyle{definition}

 \newtheorem{remark}[theorem]{Remark}
 \newtheorem{example}[theorem]{Example}
 \newtheorem{algorithm}[theorem]{Algorithm}

 \newcommand{\CC}{\ensuremath{\mathbb{C}}}

 \newcommand{\R}{\ensuremath{\mathscr{R}}}

 \newcommand{\ZZ}{\ensuremath{\mathbb{Z}}}

\newcommand{\be}{\begin{equation}}
\newcommand{\ee}{\end{equation}}

\renewcommand{\S}{\ensuremath{\mathcal{S}}}

\renewcommand{\t}[1]{\ensuremath{\tilde{#1}}}

\newcommand{\rsd}[1]{\ensuremath{\bar{#1}}}

\newcommand{\reading}{\text{\rm rowword}}

\newcommand{\sh}{\text{\rm sh}}

\newcommand{\cc}{\ensuremath{\xrightarrow{\text{cc}}}}
\newcommand{\cat}{\text{Cat}}
\newcommand{\ccat}{\text{CCat}}
\newcommand{\ccharge}{\text{cocharge}}
\newcommand{\cl}[1]{\ensuremath{{#1}^\text{cc}}}
\DeclareMathOperator{\ctype}{ctype}
\newcommand{\cinv}{\ensuremath{\mathscr{R}_{1^n}}}
\newcommand{\gd}{\ensuremath{\trianglerighteq}}
\newcommand{\ld}{\ensuremath{\trianglelefteq}}

\newcommand{\tto}{\ensuremath{\rightsquigarrow}}
\newcommand{\tleftright}{\ensuremath{\leftrightsquigarrow}}

\begin{document}
\author{Jonah Blasiak}
\address{Department of Mathematics, UC Berkeley}
\title{An insertion algorithm for catabolizability}

\begin{abstract}
Motivated by our recent work relating canonical bases to combinatorics of Garsia-Procesi modules \cite{B}, we give an insertion algorithm that computes the catabolizability of the insertion tableau of a standard word. This allows us to characterize catabolizability as the statistic on words invariant under  Knuth transformations, certain (co)rotations, and a new operation called a catabolism transformation. We also prove a Greene's Theorem-like characterization of catabolizability, and a result about how cocyclage changes catabolizability, strengthening a similar result in \cite{SW}.
\end{abstract}

\maketitle

\section{Introduction}
The ring of coinvariants $R_{1^n} = \CC[y_1, \dots, y_n] / (e_1, \dots, e_n)$,
thought of as a $\CC \S_n$-module with $\S_n$ acting by permuting the variables, is a graded version of the regular representation. It has the Garsia-Procesi modules $R_\lambda$ as quotients (see \cite{GP}).  Combining the work of Hotta-Springer and Lascoux (see \cite{H3},\cite{La},\cite{SW}) gives the Frobenius series
\be \label{e frobenius} \mathcal{F}_{R_\lambda}(t) = \sum\limits_{\stackrel{T \in SYT}{\ctype(T) \gd \lambda}} t^{\ccharge(T)} s_{\sh(T)},\ee
where SYT is the set of standard Young tableaux, and $\ctype(T)$ is the catabolizability of $T$, defined in  \textsection\ref{ss cat}.

In \cite{B} we exhibit a $q$-analogue $\cinv$ of the ring of coinvariants that is endowed with a canonical basis and
possesses $q$-analogues  $\R_\lambda$ of the $R_\lambda$ as cellular quotients. The elements of the canonical basis are in bijection with standard words (permutations of $1, \ldots, n$), the cells in bijection with SYT, and there is a natural grading on the cells that corresponds to cocharge under this bijection.  Tableau of shape $\lambda$ do not appear in degree less than $\sum_i \lambda_i(i-1)$ and there is a  unique occurrence of a tableau of shape $\lambda$ in this degree; refer to this tableau and its corresponding cell as the \emph{Garnir tableau of shape $\lambda$}.  In our investigations of the $\R_\lambda$, we found a way to go from any standard word $w$ to a word inserting to the Garnir tableau of shape $\ctype(P(w))$ by a sequence of relations in the Kazhdan-Lusztig preorder.

Following this sequence of relations gives an algorithm for computing $\ctype(P(w))$ for any standard word $w$.  This allows us to characterize catabolizability as the statistic on words invariant under Knuth transformations, non-zero (co)rotations (defined in  \textsection\ref{ss cocharge}), and a new \emph{catabolism transformation} defined in  \textsection\ref{ss CT} (and satisfying a normalization condition). We then use this to prove a Greene's Theorem-like characterization of catabolizability.

After reviewing the definitions of cocyclage and catabolism in \textsection\ref{s cocyclage and catabolism}, we present this algorithm and its corollaries in \textsection\ref{s cat algorithm}.

\section{Cocyclage and catabolism}
\label{s cocyclage and catabolism}
 We recall the notions of cocyclage and catabolizability as defined in \cite{LS}, \cite{SW}, but restrict to the special case of standard tableaux and words.
\subsection{}
We will use the following notational conventions in this paper. Tableaux are drawn using English notation, so that entries strictly increase from north to south along columns and weakly increase from west to east along rows. The notation $|z|$ denotes the length of the word $z$. For a tableau $T$, $|T|$ is the number of squares in $T$ and $\sh(T)$ is its shape.
The symbols $u, v,$ and $w$ will always denote standard words, and $y,z$ will denote words that are not necessarily standard. The notation $[a,b]$ for $a,b \in \ZZ$ denotes the set $\{i:a \leq i \leq b\}$ and $[n] := [1,n]$. 
\subsection{}
\label{ss cocharge}

The \emph{cocharge labeling} of a word $w$, denoted $\cl{w}$, is a (non-standard) word of the same length as $w$, and its numbers are thought of as labels of the numbers of $w$. It is obtained from $w$ by reading the numbers of $w$ in increasing order, labeling the 1 of $w$ with a 0, and if the $i$ of $w$ is labeled by $k$, then labeling the $i+1$ of $w$ with a $k$ (resp. $k+1$) if the $i+1$ in $w$ appears to the right (resp. left) of the $i$ in $w$. An example of a standard word and its cocharge labeling is
\be \begin{array}{ccc} w &= &1\ 6\ 8\ 4\ 2\ 9\ 5\ 7\ 3 \\ \cl{w} &= &\ 0\ 2\ 3\ 1\ 0\ 3\ 1\ 2\ 0.\end{array}\ee
The sum of the numbers in the cocharge labeling of $w$ is the \emph{cocharge} of $w$ or $\ccharge(w)$. We also set $\ccharge(\cl{w}) = \ccharge(w)$.

For a word $w$ and number $a \neq 1$, $aw$ (resp. $wa$) is a \emph{corotation} (resp. \emph{rotation}) of $wa$ (resp. of $aw$). It is a \emph{non-zero corotation} (resp. \emph{rotation}) of $wa$ (resp. of $aw$) if, in the cocharge labeling of $wa$, $a$ is labeled with a number greater than 0. Similarly, define \emph{zero (co)rotations} for the case $a$ is labeled with a 0.

(Co)rotations respect cocharge labeling in the following way:  $v$ is a corotation of $w$ if and only if $\cl{v} = a+1\ y$, $\cl{w} = ya$ for $y$ a word and $a$ a number.

Let $s_i$ be the simple reflection of $\S_n$ that transposes $i$ and $i+1$. Thinking of a standard word $w = w_1\cdots w_n$ as being the map $w: [n] \to [n]$, $i \mapsto w_i$, we can act on $w$ on the right by $\S_n$; then $w s_i$ is the word obtained from $w$ by swapping the numbers in positions $i$ and $i+1$. We can also act on cocharge labelings with this same right action, however this does not always result in the cocharge labeling of a standard word. It is not hard to see that a standard word can be recovered from its cocharge labeling. A word $z$ with $c_0$ 0's, $c_1$ 1's, $\dots$, is the cocharge labeling of some standard word if and only if there is an $l$ such that $c_i > 0$ for $i \in [0, l]$, $c_i = 0$ for $i > l$, and some $i$ appears to the left of some $i-1$ for all $i \in [l]$. The following are easily seen to be equivalent.
\be\label{e cocharge preserving}
\begin{array}{rl}
(i) & \cl{w} s_i = \cl{(w s_i)}.\\
(ii) &\cl{w} \text{ and }\cl{(w s_i)} \text{ have the same content.} \\
(iii) & \ccharge(w) = \ccharge(w s_i).\\
(iv) & |w_i - w_{i+1}| \neq 1.
\end{array}
\ee
If any (all) of these holds, then the transformation $w \tto w s_i$ is \emph{cocharge-preserving}.

There is a \emph{cocyclage} from the tableau $T$ to the tableau $T'$, written $T \cc T'$, if there exist words $u, v$ such that $v$ is the corotation of $u$ and $P(u) = T$ and $P(v) = T'$. The \emph{cocyclage poset} is the poset on the set of SYT generated by the relation \cc.

Define the cocharge labeling $\cl{T}$ of a tableau $T$ to be $P(\cl{\reading(T)})$, and $\ccharge(T)$ to be the sum of the entries in $\cl{T}$. Here $\reading(T)$ denotes the row reading word of $T$. The tableau $\cl{T}$ is also $P(\cl{w})$ for any $w$ inserting to $T$. This follows from the fact that Knuth transformations do not change left descent sets. A cocyclage $T \cc T'$ is a \emph{non-zero (resp. zero) cocyclage} if any (equivalently, every) corotation inducing it is non-zero (resp. zero).

The following theorem is easy for the standard tableaux case. The \emph{cyclage poset} is the poset dual to the cocyclage poset, i.e., the poset obtained by reversing all relations.

\begin{theorem}[\cite{LS}]
The cyclage poset is graded, with rank function given by cocharge.
\end{theorem}

\subsection{}
\label{ss cat}

Let $Z_\lambda$ be the superstandard tableau of shape $\lambda$ with $0$'s in the first row, $1$'s in the second row, etc., and let $Z^*_\lambda$ be the standard tableau such that $\cl{Z^*_\lambda} = Z_\lambda$. For a skew tableau $T$ and index $r$ (resp. index $c$), let $H_r(T) = P(T_nT_s)$ (resp. $V_c(T)=P(T_eT_w))$, where $T_n$ and $T_s$ (resp. $T_e$ and $T_w$) are the north and south (resp. east and west) subtableaux obtained by slicing $T$ horizontally (resp. vertically) between its $r$-th and $(r+1)$-th rows (resp. $c$-th and $(c+1)$-th columns).
Recall that we are using English notation for tableaux.

For $\lambda \subseteq \sh(T)$, let $T_\lambda$ be the subtableau of $T$ of shape $\lambda$. If $(m) \subseteq \sh(T)$, then define the $m$-catabolism (resp. $m$-column catabolism) of $T$, notated $\cat_m(T)$ (resp. $\ccat_m(T)$), to be the tableau $H_1(T-T_{(m)})$ (resp. $V_m(T-T_{(m)}$). For a partition $\lambda \vdash n := |T|$, $\lambda$-catabolizability (resp. $\lambda$-column catabolizability) is defined inductively as follows: $T$ is $\lambda$-(column) catabolizable if $T_{(\lambda_1)}$ contains the $\lambda_1$ smallest entries of $T$ and the $\lambda_1$-(column) catabolism of $T$ is $\widehat{\lambda}$-(column) catabolizable, where $\lambda = (\lambda_1, \widehat{\lambda})$; the empty tableau is $\varnothing$-(column) catabolizable.

The following is a consequence of Proposition 48 in \cite{SW} and the surrounding discussion.

\begin{proposition}
\label{p basic cat}
For any SYT $T$, there is a unique maximal in dominance order partition $\lambda$ such that $T$ is $\lambda$-catabolizable.
\end{proposition}

The $\lambda$ of this proposition is the \emph{catabolizability} of $T$, denoted $\ctype(T)$.  Catabolizability is  computed by performing the sequence of catabolisms to $T$ in which $\cat_m$ is applied with the largest $m$ such that $T_{(m)} = Z^*_{(m)}$. Similarly, define the column catabolizability of $T$ to be the partition obtained by performing the sequence of column catabolisms to $T$ in which $\ccat_m$ is applied with the largest $m$ such that $T_{(m)} = Z^*_{(m)}$.  With the definitions of $\lambda$-(column) catabolizability above, the following proposition is quite tricky.

\begin{proposition}[{\cite[Proposition 49]{SW}}]
\label{p row cat equals column cat}
A standard tableau is $\lambda$-catabolizable if and only if it is $\lambda$-column catabolizable.
\end{proposition}

We reprove this result in the next section using the catabolism insertion algorithm.

\section{Catabolism insertion}
\label{s cat algorithm}
\subsection{}
\label{ss CT}
We now describe the \emph{catabolism insertion algorithm}, which takes a standard word $w$ as input and outputs a partition $F(w)$ that we will show to be equal to the catabolizability of the insertion tableau of $w$.

Let $\epsilon_i \in \ZZ^n$ be the standard basis vector with a 1 in its $i$-th coordinate and $0$'s elsewhere.

\begin{algorithm}\label{cat_algorithm}
Let $f$ be the function below, which takes a pair consisting of a (non-standard) word and a partition to another such pair. Let $x = ya$, $y$ a word and $a$ a number.

\be \label{e f}
f(x, \nu) =
\begin{cases}
(y, \nu + \epsilon_{a+1}) & \text{if } \nu + \epsilon_{a+1} \text{ is a partition,} \\
(a+1 \ y, \nu) & \text{otherwise.}
\end{cases} \ee

Given the input standard word $w$, first determine the cocharge labeling $z$ of $w$.

Next, apply $f$ to $(z, \emptyset)$ repeatedly, obtaining the sequence of pairs $f^{(i)}(z,\emptyset)$, stopping when the word of the pair is empty. Output the partition of this final pair, and denote this output $F(w)$.

The transition from $(x, \nu)$ to $f(x, \nu)$ is a \emph{step} of the algorithm.   We say that $a$ is \emph{presented} to $\nu$, and in the top case of (\ref{e f}), $a$ is \emph{inserted} into $\nu$, while in the bottom case, $a$ is \emph{corotated}.  The step of the algorithm in the top case is an \emph{insertion}, and in the bottom a  \emph{corotation}.
\end{algorithm}

It is convenient to think of the $\nu$ in the algorithm as the tableau $Z_\nu$, as illustrated by the following example.

\begin{example}
The word $w=1\ 6\ 8\ 4\ 2\ 9\ 5\ 7\ 3$ has cocharge labeling $z = 0\ 2\ 3\ 1\ 0\ 3\ 1\ 2\ 0$. The sequence of word-partition pairs produced by the algorithm is

\[  \Yboxdim9pt \small\begin{array}{rrlcrrl}
i\quad & \multicolumn{2}{c}{f^{(i)}(z,\emptyset)} & \qquad\qquad& i\quad & \multicolumn{2}{c}{f^{(i)}(z,\emptyset)} \\ \\
0\quad & 023103120 & {\tiny\emptyset} &\qquad\qquad&7\quad& 44302 & \tiny\young(00,11) \\ \\
1\quad  & 02310312 & \tiny\young(0) &\qquad\qquad&8\quad&4430 & \tiny\young(00,11,2) \\ \\
2\quad & 30231031 & \tiny\young(0) &\qquad\qquad&9\quad&443 & \tiny\young(000,11,2) \\ \\
3\quad & 3023103 & \tiny\young(0,1) &\qquad\qquad&10\quad& 44 & \tiny\young(000,11,2,3) \\ \\
4\quad & 4302310 & \tiny\young(0,1) &\qquad\qquad&11\quad&4 & \tiny\young(000,11,2,3,4) \\ \\
5\quad & 430231 & \tiny\young(00,1) &\qquad\qquad&12\quad&5 & \tiny\young(000,11,2,3,4) \\ \\
6\quad & 43023 & \tiny\young(00,11) &\qquad\qquad&13\quad& \emptyset & \tiny\young(000,11,2,3,4,5)
\end{array} \]
\end{example}

Letting $z$ be the cocharge labeling of $w$, a \emph{catabolism transformation} of $w$ (or of $z$) is an operation taking $w$ to $w s_i$ ($z$ to $z s_i$) if $|z_i - z_{i+1}| > 1$:
\be\label{e up arrow} \cdots z_i z_{i+1} \cdots \tleftright\ \cdots z_{i+1} z_i \cdots.\ \ee
It is easy to see from (\ref{e cocharge preserving}) (iv) that a catabolism transformation is cocharge-preserving.

We also define an \emph{ascent} of  $w$ (or of $z$) to be a transformation of the form $z \tto z s_i$ provided it is cocharge-preserving and $z_i > z_{i+1}$.

Some properties of Algorithm \ref{cat_algorithm} are more easily seen from the following variant, which is clearly equivalent to it.
\begin{algorithm}\label{cat_algorithm2}
A step of Algorithm \ref{cat_algorithm} from $(ya,\nu)$ to $f(ya,\nu)$ is rephrased as follows. Instead of keeping track of a partition $\nu$, keep track of the corresponding superstandard tableau $Z_\nu$.  Replace presenting $a$ to $\nu$ with column-inserting $a$ into $Z_\nu$.
In the insertion case, this produces the same result as in Algorithm \ref{cat_algorithm}.

The corotation case is broken into three parts, the first of which is this column-insertion; let $T$ be the tableau resulting from this insertion. The tableau $T$ contains $Z_\nu$ and $T - Z_\nu$ is a single square containing an $a$.   The corresponding number of $\reading(T)$ is at least two more than all the numbers to the right of it.  The second part performs the sequence of catabolism transformations taking $\reading(T)$ to $\reading(Z_\nu) \ a$.   The third part then corotates
\[ y\ \reading(Z_\nu)\ a \]
to obtain
\[ a+1\ y\ \reading(Z_\nu), \]
and this word is output as the pair $(a+1 \ y, Z_\nu)$.
\end{algorithm}

\begin{theorem}\label{cat_algorithm_theorem}
Algorithm \ref{cat_algorithm}, with input a standard word $u$ of length $n$, satisfies:
\begin{list}{\emph{(\roman{ctr})}} {\usecounter{ctr} \setlength{\itemsep}{1pt} \setlength{\topsep}{2pt}}
\item after every step, the word-partition pair $(x,\nu)$ is such that \linebreak $x\ \reading(Z_\nu)$ is the cocharge labeling of some standard word,
\item terminates (in at most $n+\binom{n}{2}$ steps),
\item $F(u) = F(v)$ if $v$ is a non-zero corotation of $u$,
\item $F(u) = F(v)$ if $u \tto v$ is a catabolism transformation,
\item $F(u) = F(v)$ if $u \tto v$ is a Knuth transformation,
\item $F(u) =  \ctype(P(u))$.
\end{list}
\end{theorem}
\begin{proof}
Statement (i) is easy to see from Algorithm \ref{cat_algorithm2}.  It realizes each step as a sequence of Knuth transformations, catabolism transformations, or non-zero (co)rotations on the word $\ x\ \reading(Z_\nu)$, all of which could be applied by converting back to a standard word, performing the operation, and then taking the cocharge labeling.

The algorithm terminates in at most $n+\binom{n}{2}$ steps because for any step from $(x,\nu)$ to $(x^*,\nu^*) := f(x,\nu)$, either \begin{multline*}
|x^*| = |x| - 1  \quad\text{ or } \\\ccharge(x^* \reading(Z_{\nu^*})) = \ccharge(x\ \reading(Z_\nu)) + 1;
 \end{multline*}
 cocharge never exceeds $\binom{n}{2}$.

Statement (iii) is clear.

For (iv), suppose that $\cl{u} \tto \cl{v}$ is the catabolism transformation
\be \cdots ab\cdots\tto \cdots ba\cdots.\ee Whether $a$ is corotated or inserted does not depend on what happens to $b$ and similarly whether $b$ is corotated or inserted does not depend on what happens to $a$. Therefore, after $a$ and $b$ are presented, the resulting pair $(x, \nu)$ is the same for the algorithm applied to $u$ and applied to $v$, unless $a$ and $b$ are both corotated. In this case, the result follows by induction since we can apply the same argument to the catabolism transformation $\cdots a+1\ b+1\cdots \tto \cdots b+1\ a+1\cdots$.

To show (v), we prove the slightly stronger statement that, after any step of the algorithm, the pair $(x,\nu)$ can be replaced by $(x',\nu)$ with $x \tto x'$ a Knuth transformation without changing the output.  We may assume that $x \tto x'$ is a Knuth transformation in the last three numbers of $x$ and is written as
\be \cdots bac \tto \cdots bca.\ee 
Given (iv), we may further assume that $b=c$, $a=b-1$ (we must also check the case $\cdots acb \tto \cdots cab$, where $b=a$ and $c=b+1$, but this is similar). If $a > 0$ and $\nu_a = \nu_{a+1}=\nu_{b}=\nu_{b+1}$, then  $a,b$, and $c $ are corotated and the result follows by induction from the Knuth transformation \be b+1\ a+1\ c+1\cdots\tto b+1\ c+1\ a+1 \cdots. \ee Otherwise, one checks that  $f^{(3)}(x,\nu)=f^{(3)}(x',\nu)$.

Given (v), after any step, we are free to replace $x$ with something Knuth equivalent to it without changing the final output. Also, after any step, we may run the algorithm on anything Knuth equivalent to $x\ \reading(Z_\nu)$ and get the same output. Thus for (vi), run the algorithm on $\reading(P(u))$. The first steps of the algorithm corotate the numbers in the first row of $\cl{P(u)}$ that are not $0$. The next $m$ steps are insertions of $0$'s, where $m$ is the number of $0$'s in $\cl{P(u)}$, also the largest integer such that $P(u)_{(m)} = Z_{(m)}^*$. Let $(x, \nu)$ be the pair at this stage, $x'$ be the result of subtracting $1$ from all numbers in $x$, and  $v$ be the standard word with $\cl{v}=x'$. The result then follows from
\be F(u) = (m, F(v)) = (m, \ctype(P(v))) = \ctype(P(u)). \ee 
The leftmost equality holds because the word $x$ contains at most $m$ $1$'s, and the middle equality is the inductive statement $F(v) = \ctype(P(v))$.  The rightmost equality holds because the standard tableau  $P(v)$ is the same as the result of subtracting $m$ from all entries in $\cat_m(P(u))$ and $\nu = (m)$.

\end{proof}

From now on, we write $\ctype(u)$ for $\ctype(P(u))$. From the theorem and remarks in the proof of (i), we obtain the following.

\begin{corollary}
\label{c cke}
Catabolizability is characterized as the statistic on standard words that is invariant under non-zero (co)rotations, catabolism transformations, and Knuth transformations, and satisfies
\[ \ctype(\reading(Z^*_\lambda)) = \lambda. \]
\end{corollary}

A similar proof to that of (vi) gives
\begin{corollary}
\label{c row catability eq column catability}
The catabolizability of a tableau $T$ equals the column catabolizability of $T$.
\end{corollary}

\subsection{}
We present one more variant of the catabolism insertion algorithm that determines whether a tableau is $\lambda$-catabolizable rather than its catabolizability.

\begin{algorithm}
\label{cat_algorithm3}
This takes a standard word $w$ of length $n$ and a partition $\lambda$ of $n$ as input. It is the same as Algorithm \ref{cat_algorithm} except that the partition $\nu$ of the word-partition pair is forced to satisfy $\nu_i \leq \lambda_i$ for all $i$. Precisely, $f$ is replaced by

\be \label{e flambda}
f_\lambda(ya, \nu) =
\begin{cases}
(y, \nu + \epsilon_{a+1}) & \text{if } \nu + \epsilon_{a+1} \text{ is a partition and $\nu_{a+1}+1 \leq \lambda_{a+1}$} \\
(a+1 \ y, \nu) & \text{otherwise.}
\end{cases} \ee

This algorithm repeatedly applies $f_\lambda$ to $(\cl{w}, \emptyset)$ until either a number $a$ is presented such that $\lambda_{a+1}=0$, or the word of the pair is empty. The algorithm outputs false if the former occurs and true if the latter occurs.
\end{algorithm}

\begin{theorem}\label{cat_algorithm_theorem2}
Algorithm \ref{cat_algorithm3}, with input a standard word $u$ of length $n$ and a partition $\lambda$ of $n$, outputs true if and only if $P(u)$ is $\lambda$-catabolizable.
\end{theorem}
\begin{proof}
The proof is essentially the same as the proof of Theorem \ref{cat_algorithm_theorem}. The most significant difference is for the proof of (i).
The only case in which Algorithms \ref{cat_algorithm} and \ref{cat_algorithm3} differ is when $a$ is presented and $\nu + \epsilon_{a+1}$ is a partition but $\nu_{a+1}+1 > \lambda_{a+1}$. Such a corotation step of Algorithm \ref{cat_algorithm3} can also be broken into several steps as in Algorithm \ref{cat_algorithm2}.  If $\lambda_{a+1} >0$, then this step can be broken into a sequence of Knuth transformations, ascents, and catabolism transformations followed by a corotation.  These operations could be applied by converting back to a standard word, performing the operation, and then taking the cocharge labeling.  If on the other hand, $\lambda_{a+1 }=0 $, then the algorithm terminates and outputs false.
\end{proof}

A corollary to this theorem is Proposition \ref{p row cat equals column cat} (Proposition 49 in \cite{SW}).  
\begin{remark}
We could have also concluded Proposition \ref{p row cat equals column cat} from Corollary \ref{c row catability eq column catability} and Proposition \ref{p basic cat} and its version for column catabolizability, however Proposition \ref{p basic cat} for column catabolizability  is significantly more difficult than the proposition itself. Instead, Proposition \ref{p row cat equals column cat} and Proposition \ref{p basic cat} allow us to conclude Proposition \ref{p basic cat} for column catabolizability, although a roundabout way.
\end{remark}
\subsection{}
\label{ss Greene}
Corollary \ref{c cke} allows us to characterize catabolizability in a similar way to the Greene's Theorem interpretation of the shape of the insertion tableau of a word (see \cite[Lemma A1.1.7]{F}). This is reminiscent of the combinatorial description of two-sided cells in the affine Weyl group of type $A$ (see, for instance, \cite{X}), and also of the usual way of computing cocharge of semistandard words.

Let $w$ be a standard word and $z$ its cocharge labeling. Define $\t{w}: \ZZ_{\leq n} \to \ZZ_{\geq 0}$ by $i \mapsto z_{i + kn} + k$, where $k$ is the unique integer so that $i + kn \in [n]$. Also let $\t{z}$ refer to this same map. Let $\rsd{}: \ZZ_{\leq n} \to \ZZ/n\ZZ$ be the map sending an integer to its congruence class mod $n$.

A \emph{chain} of $\t{w}$ of \emph{length} $k'+1$ is a sequence $\mathbf{j} = (j_{k'}, j_{k'-1}, \dots, j_0)$  satisfying
\begin{list}{\emph{(\roman{ctr})}} {\usecounter{ctr} \setlength{\itemsep}{1pt} \setlength{\topsep}{2pt}}
\item $j_{k'} < j_{k'-1} < \dots < j_0$,
\item $\t{w}(j_i) = i$, for all $i \in [0, k']$,
\item $\rsd{j_i} \neq \rsd{j_{i'}}$, for all $i, i' \in [0, k']$ with $i \neq i'$.
\end{list}
The underlying set of $\mathbf{j}$ is $\mathbf{j}^* := \{ j_{k'}, j_{k'-1}, \dots, j_0 \}$.

A \emph{$k$-bounded chain family} of $\t{w}$ is a set $\mathfrak{A} = \{A_1, \dots, A_l\}$ of chains of $\t{w}$ of lengths at most $k$ such that the subsets $\rsd{A_i}$ of $\ZZ/n\ZZ$ are disjoint. The \emph{support} of $\mathfrak{A}$ is $\cup_i A_i^* \subseteq \ZZ_{\leq n}$ and the \emph{size} of a $k$-bounded chain family is the cardinality of its support. The maximum size of a $k$-bounded chain family of $\t{w}$ is denoted $I_k(\t{w})$.

Let $s_d: \ZZ_{\leq n} \to \ZZ_{\leq n}$ for $d \in [n-1]$ be the affine versions of the simple reflections of $\S_n$: $s_d$ transposes $d+kn$ and $d+1+kn$ for all $k \leq 0$. Certainly $\t{u}s_d = \widetilde{u s_d}$ if $u \tto u s_d$ is cocharge-preserving. The action of $\S_n$ on chains is given by $s_d(\mathbf{j}) = ( s_d(j_{k'}), s_d(j_{k'-1}), \dots, s_d(j_0))$.

Now assume that standard words $u$ and $v$ differ from each other by a simple transposition, $v = u s_d$, so that the transformation $u \tto v$ is cocharge-preserving.

\begin{lemma}
\label{l chain}
With $u, v$ as above, suppose $\cl{u}_d \neq \cl{u}_{d+1} + 1$. Then, if $\mathbf{j}$ is a chain of $\t{u}$, then $s_d(\textbf{j})$ is a chain of $\t{v}$.
\end{lemma}

\begin{proof}
First observe
\be \t{v}(s_d(\mathbf{j})) = \t{u}(s_d(s_d(\mathbf{j}))) = \t{u}(\mathbf{j}). \ee
The only way $s_d(\mathbf{j})$ is not a chain of $\t{v}$ is if $j_i = j_{i-1} - 1 \equiv d \mod n$ for some $i$, which is excluded by the assumption $\cl{u}_d \neq \cl{u}_{d+1} + 1$.
\end{proof}

\begin{theorem}
\label{t green}
With the notation above,
\be \sum\limits_{i=1}^{k} \ctype(w)_i = I_k(\t{w}). \ee
\end{theorem}

\begin{proof}
By Corollary \ref{c cke}, it suffices to check that $I_k(\t{w})$ is $\sum\limits_{i=1}^{k} \lambda_i$ for $\cl{w} = \reading(Z_\lambda)$, is invariant under non-zero (co)rotations, catabolism transformations, and Knuth transformations.

For $\cl{w} = \reading(Z_\lambda)$, $I_k(\t{w}) = \sum\limits_{i=1}^{k} \lambda_i$. There holds
\[|\{\rsd{i} : \t{w}(i) \leq k-1\}| = \sum\limits_{i=1}^{k} \lambda_i \] so that $I_k(\t{w})$ cannot possibly exceed this number. It is easy to exhibit a $k$-bounded chain family of $\t{w}$ of size $\sum\limits_{i=1}^{k} \lambda_i$.

Non-zero (co)rotations. If $v$ is a non-zero corotation of $u$, then $\t{v}(i) = \t{u}(i-1)$ and $\t{u}(n) \neq 0$. Thus the support of a $k$-bounded chain family of $\t{u}$ cannot contain $n$ and the notions of a $k$-bounded chain family for $\t{u}$ and $\t{v}$ differ only by shifting indices by 1.

Catabolism transformations. This follows from the special case of Lemma \ref{l chain} in which $|\cl{u}_d - \cl{u}_{d+1}| \neq 1$.

Knuth transformations. We may assume that $u \tto v$ is a Knuth transformation with $v = u s_d$ and $\cl{u}_d = a + 1$, $\cl{u}_{d+1} = a$, $\cl{u}_{d+2} = a$.
\be u = \cdots a+1\ a \ a \cdots \tto v = \cdots a\ a+1\ a\ \cdots \ee
By Lemma \ref{l chain} and its proof, any $k$-bounded chain family of $\t{v}$ yields one of the same size for $\t{u}$, and any $k$-bounded chain family $\mathfrak{A}$ of $\t{u}$ yields one of the same size for $\t{v}$ provided any $\mathbf{j} \in \mathfrak{A}$ does not satisfy $j_i = j_{i-1} - 1 \equiv d \mod n$ for some $i$. If this is the case, then one checks that $\{s_d s_{d+1}(\mathbf{j}) : \mathbf{j} \in \mathfrak{A}\}$ is a $k$-bounded chain family of $\t{v}$, clearly of the same size as $\mathfrak{A}$.
\end{proof}

This proof differs in an important way with that of the Greene's Theorem interpretation of insertion tableau \cite[Lemma A1.1.7]{F}. None of the alterations of chains in the proof change the lengths of chains. This allows us to conclude the stronger statement:

\begin{theorem}
If $\ctype(w) = \lambda$ then $\t{w}$ has an $\ell(\lambda)$-bounded chain family consisting of chains of lengths $\lambda'_1, \lambda'_2, \dots, \lambda'_{\lambda_1}$.
\end{theorem}

Applying Lemma \ref{l chain} and Theorem \ref{t green}, we obtain

\begin{corollary}
\label{c ascent edge}
If $u \tto v$ is an ascent, then $\ctype(u) \gd \ctype(v)$.
\end{corollary}

The next corollary is a strengthening of \cite[Lemma 51]{SW}.

\begin{corollary}
\label{c corotation}
If $v$ is a non-zero corotation of $u$, then $\ctype(v) = \ctype(u)$. If $v$ is a zero corotation of $u$, then $\ctype(v)$ is strictly less than $\ctype(u)$ in dominance order.
\end{corollary}

\begin{remark}
It is tempting to conjecture that if  $u \tto v$ is an ascent or a zero corotation and $\ctype(v) \neq \ctype(u)$, then $\ctype(v) \ld \ctype(u)$ is a covering relation in the dominance order poset. This is false. For example, if $u = 0\ 0\ 1\ 1\ 0$,  $v= 1\ 0\ 0\ 1\ 1 $, then $\ctype(u)=(3, 1, 1)$, which does not cover   $\ctype(v)=(2, 1, 1, 1)$; if $u= 1\ 1\ 2\ 2\ 0\ 0\ 1\ 0$, $v=1\ 1\ 2\ 2\ 0\ 0\ 0\ 1 $, then $\ctype(u)=(3, 3, 1, 1)$, which does not cover $\ctype(v)=(3, 2, 1, 1, 1)$.
\end{remark}

\section*{Acknowledgments}
This paper would not have been possible without the generous advice from and many detailed discussions with Mark Haiman.  I am also grateful to Michael Phillips and Ryo Masuda for help typing and typesetting figures.

 \end{document}